\documentclass[11pt]{amsart}
\usepackage{amsmath,amsthm,amsfonts,amscd,amssymb,eucal,latexsym,mathrsfs,enumerate,framed}

\usepackage{hyperref}

\newcommand{\BHs}{\mathcal B(\mathcal H)}
\newcommand{\BKs}{\mathcal B(\mathcal K)}
\newcommand{\Hs}{\mathcal H}
\newcommand{\Ks}{\mathcal K}

\newcommand{\id}{\mathrm{id}}
\newcommand{\Ad}{\mathrm{Ad}}

\newcommand{\eps}{\varepsilon}

\newcommand{\vnotimes}{\,\overline{\otimes}\,}

\theoremstyle{plain}
\newtheorem{lemma}{Lemma}
\newtheorem{theorem}[lemma]{Theorem}
\newtheorem{corollary}[lemma]{Corollary}
\newtheorem{proposition}[lemma]{Proposition}
\newtheorem{conjecture}[lemma]{Conjecture}

\theoremstyle{definition}
\newtheorem{definition}[lemma]{Definition}
\newtheorem{question}{Question}

\title{Type II$_1$ factors satisfying the spatial isomorphism conjecture}
\author[Cameron]{Jan Cameron}
\address{\hskip-\parindent
Jan Cameron, Department of Mathematics, Vassar College, Poughkeepsie, New York,
U.S.A.}
\email{jacameron@vassar.edu}
\thanks{JC's research is partially supported by an AMS-Simons research travel grant.}
\author[Christensen]{Erik Christensen}
\address{\hskip-\parindent
Erik Christensen, Institute for Mathematiske Fag, University of Copenhagen, Copenhagen, Denmark.}
\email{echris@math.ku.dk}
\author[Sinclair]{Allan M.~Sinclair}
\address{\hskip-\parindent
Allan M.~Sinclair, School of Mathematics, University of Edinburgh, JCMB, King's Buildings, Mayfield Road, Edinburgh, EH9 3JZ, Scotland.}
\email{a.sinclair@ed.ac.uk}
\author[Smith]{Roger R.~Smith}
\address{\hskip-\parindent
Roger R.~Smith, Department of Mathematics, Texas A{\&}M University,
College Station, Texas, 77843, U.S.A.}
\email{rsmith@math.tamu.edu}
\thanks{RS is partially supported by NSF grant DMS-1101403}

\author[White]{Stuart White}
\address{\hskip-\parindent
Stuart White, School of Mathematics and Statistics, University of Glasgow, 
University Gardens, Glasgow Q12 8QW, Scotland.}
\email{stuart.white@glasgow.ac.uk}
\thanks{SW is partially supported by EPSRC grant EP/I019227/1}

\author[Wiggins]{Alan D.~Wiggins}
\address{\hskip-\parindent
Alan D.~Wiggins, Department of Mathematics and Statistics. University of Michigan-Dearborn, 4901 Evergreen Road, Dearborn, Michigan, 48126, USA.}
\email{adwiggin@umd.umich.edu}

\date{\today}

\begin{document}
\maketitle
\begin{abstract} This paper addresses a conjecture of Kadison and Kastler that a von Neumann algebra $M$ on a Hilbert space $\Hs$ should be unitarily equivalent to each sufficiently  close von Neumann algebra $N$ and, moreover, the implementing unitary can be chosen to be close to the identity operator. This is known to be true for amenable von Neumann algebras and in this paper we describe new classes of non-amenable factors for which the conjecture is valid. These are based on tensor products of the hyperfinite II$_1$ factor with crossed products  of abelian algebras by suitably chosen discrete groups. 
\end{abstract}

In 1972, Kadison and Kastler initiated the study of
perturbation theory of operator algebras, \cite{KK:AJM}. The setting was a Hilbert space $\Hs$
and the collection of all von Neumann subalgebras of the bounded operators
$\BHs$ on $\Hs$, namely those $^*$-closed subalgebras of $\BHs$ which contain the identity operator and are closed in the strong operator topology.  By applying the Hausdorff distance to the unit balls of two von Neumann algebras, they equipped the collection of all von Neumann subalgebras with a metric $d(\cdot,\cdot)$. This can be
described as the infimum of numbers $\lambda >0$ for which each element of
either unit ball is within a distance $\lambda$ of an element of the other in the operator norm on $\BHs$.  Naturally examples of close pairs of von Neumann algebras arise by fixing a von Neumann algebra $M\subseteq\BHs$ and considering a unitary $u\in\BHs$.  It is easy to see that 
$$
d(M,uMu^*)\leq 2\|u-\id_\Hs\|
$$
and so if $u$ is chosen with $\|u-\id_\Hs\|$ small, then $uMu^*$ will be close to $M$. In this case we refer to $uMu^*$ as a \emph{small unitary perturbation} of $M$.  In \cite{KK:AJM}, Kadison and Kastler proposed that such a small unitary perturbation should be essentially the only way of producing pairs of close von Neumann algebras, leading to the following conjecture.

\begin{conjecture}[Kadison-Kastler]\label{C-KK}
For all $\eps>0$, there exists $\delta>0$ with the property that if $M,N\subseteq\BHs$ are von Neumann algebras with $d(M,N)<\delta$, then there exists a unitary operator $u$ on $\Hs$ with $uMu^*=N$ and $\|u-\id_\Hs\|<\eps$.
\end{conjecture}

Initial progress on this conjecture focused on amenable von Neumann algebras which, due to the work of Connes, \cite{Con:Ann}, may be characterized as inductive limits of finite dimensional von Neumann algebras. 
For these algebras, this conjecture was established in the late 1970's by E.C., \cite{C:Invent}, Johnson, \cite{Joh:PLMS}, and Raeburn and Taylor, \cite{RT:JFA}, (see Theorem \ref{T-C} below). In this paper we will describe our examples of the first non-amenable von Neumann algebras which satisfy the conjecture. Full details and proofs can be found in the longer account \cite{CCSSWW:Preprint}.

\section*{Background}

As the Kadison-Kastler conjecture predicts that close operator algebras should be isomorphic, it is natural to ask whether they necessarily share the same invariants and structural properties. This was the primary focus of \cite{KK:AJM}, which examined the \emph{type decomposition} of close von Neumann algebras.  The foundational work of Murray and von Neumann decomposes every von Neumann algebra $M$ uniquely into a direct sum $M_{\mathrm{I}}\oplus M_{\mathrm{II}_1}\oplus M_{\mathrm{II_\infty}}\oplus M_{\mathrm{III}}$, where the summands have types I, II$_1$, II$_\infty$ and III respectively. In particular, every von Neumann \emph{factor} (those von Neumann algebras which are maximally non-commutative in that the centers consist only of scalar multiples of the identity operator) is of one of these types. Our work is concerned with factors of type II$_1$, and a
formulation equivalent to the original definition is that $M$ should be
infinite dimensional and possess a positive linear functional $\tau$ of  norm 1
satisfying $\tau(ab)=\tau(ba)$ for $a,b\in M$. This functional is called a
trace, and is  the counterpart of the standard trace on the algebra
of $n\times n$ matrices that averages the diagonal entries.   The main theorem of \cite{KK:AJM} shows that if $M$ and $N$ are close von Neumann algebras, then the projections onto the summands of each type are necessarily close. This work also shows that algebras close to factors are again factors, and so any von Neumann algebra close to a II$_1$ factor is again a II$_1$ factor, a result we will use subsequently.  

It is also natural to consider perturbation theory for other classes of operator algebras. In \cite{Ph:IUMJ}, Phillips initiated the study of these questions in the context of norm closed self-adjoint algebras ($C^*$-algebras), and examined the ideal lattices of close algebras.  A key difference in flavor between perturbation theory for $C^*$-algebras and the von Neumann algebra version was exposed in two critical examples: \cite{CC:BLMS} gives examples of arbitrarily close but non-isomorphic $C^*$-algebras, while \cite{Joh:CMB} gives examples of close, unitarily conjugate separable $C^*$-algebras for which it is not possible to choose a unitary witnessing this conjugacy close to the identity. The counterexamples of \cite{CC:BLMS} are non-separable, so the appropriate formulation of Conjecture \ref{C-KK} for $C^*$-algebras is that sufficiently close separable $C^*$-algebras acting on a separable Hilbert space should be spatially isomorphic, but without asking for control of the unitary implementing a spatial isomorphism. Special cases of this conjecture were established for separable approximately finite dimensional $C^*$-algebras, \cite{PhR:CJM,C:Acta}, and continuous trace algebras, \cite{PhR:PLMS}, in the early 1980's, and a complete analogue of the perturbation results for amenable von Neumann algebras was recently given in \cite{CSSWW:PNAS,CSSWW:Acta} which establishes the conjecture for separable nuclear $C^*$-algebras.  There has also been significant work on perturbation questions for non-self-adjoint algebras, see \cite{Pit:MA} for example.

A related notion of near
containments also plays a substantial role in our work. We say that $M\subseteq_{\gamma}
N$ if each element of the unit ball of $M$ is within a distance $\gamma$ of an
element of $N$ (not required to be in the unit ball of $N$).  Analogously to Conjecture \ref{C-KK}, one might expect a sufficiently small near inclusion of von Neumann algebras to arise from a small unitary perturbation of a genuine inclusion. That is, for each $\eps>0$, does there exist $\delta>0$ such that if $M\subseteq_\delta N$ is a near inclusion of von Neumann algebras on $\Hs$, then there is a unitary $u$ on $\Hs$ with $uMu^*\subseteq N$ and $\|u-\id_\Hs\|<\eps$?  E.C. introduced this notion in \cite{C:Acta},
with the twofold purpose of improving numerical estimates and of extending perturbation results beyond the amenable von Neumann algebra setting.  In particular E.C. gave the following positive answer to the previous question when $M$ is amenable but $N$ is arbitrary. It is easy to use Theorem \ref{T-C} to show that if $d(M,N)<1/101$ and $M$ is amenable, then there is a unitary $u\in (M\cup N)''$ with $uMu^*=N$ and $\|u-\id_\Hs\|\leq 150d(M,N)$.  

\begin{theorem}[Spatial embedding theorem]\label{T-C}
Let $M$ and $N$ be von Neumann algebras on a Hilbert space $\Hs$ and suppose
that $M$ is amenable. If $M\subset_\gamma N$ for a constant $\gamma<1/100$, then there
exists a unitary $u\in (M\cup
N)''$ so that $\|u-\id_\Hs\|\leq150\gamma$, $d(M,uMu^*)\leq100\gamma$ and $uMu^*\subseteq N$.
\end{theorem}
Embedding theorems are also possible in the setting of $C^*$-algebras; given a sufficiently close near inclusion of a separable nuclear $C^*$-algebra $A$ into a general $C^*$-algebra $B$,  \cite{HKW:Adv} establishes the existence of an embedding $A\hookrightarrow B$.

The other general context in which perturbation results have been obtained is when we replace $\BHs$ with a finite von Neumann algebra.  Given unital von Neumann subalgebras $B_1$ and $B_2$ of a finite von Neumann algebra $M$ with $d(B_1,B_2)<1/8$, \cite{C:IJM} gives a unitary $u\in (B_1\cup B_2)''$ with $uB_1u^*=B_2$ and $\|u-1_M\|\leq7d(B_1,B_2)$. 

In our longer account \cite{CCSSWW:Preprint} of the work surveyed in this paper, we keep track of the estimates involved at each step.  Here, we simplify matters by describing our results qualitatively.

\section*{Kadison-Kastler stability and the similarity problem}

The spatial embedding theorem does not depend on the particular $^*$-representation of $M$ on a Hilbert space. Our search for positive answers to Conjecture \ref{C-KK} is guided by this result, leading us to the following definition.

\begin{definition}
Let $M$ be a von Neumann algebra. Say that $M$ is \emph{strongly Kadison-Kastler stable} if, for every $\eps>0$, there exists $\delta>0$ such that for every faithful normal unital $^*$-representation $\pi:M\rightarrow\BHs$, and every von Neumann algebra $N$ on $\Hs$ with $d(\pi(M),N)<\delta$, there is a unitary operator $u$ on $\Hs$ with $u\pi(M)u^*=N$ and $\|u-\id_\Hs\|<\eps$.
\end{definition}
We use this terminology  as this is the strongest of several versions of the conjecture that one could consider. For example, one could ask for  spatial isomorphisms without requiring control of $\|u-\id_\Hs\|$, or for isomorphisms between close algebras which are not necessarily spatial. Our methods also give examples of von Neumann algebras satisfying these weaker forms of the conjecture, see Theorems \ref{T-Free} and \ref{T-Hyp} below.  An $\ell^\infty$-direct sum argument can be used to show that Conjecture \ref{C-KK} is equivalent to the statement that all von Neumann algebras are strongly Kadison-Kastler stable.

Conjecture \ref{C-KK} implies that the operation
$$
M\mapsto M'=\{x\in \BHs:\forall y\in M,\ xy=yx\}
$$
of taking commutants of von Neumann algebras on $\BHs$ is continuous with respect to the Kadison-Kastler metric, and this would extend to $C^*$-algebras by an application of Kaplansky's density theorem. This is equivalent to another long standing question: the similarity problem.  In 1955, motivated by work on Dixmier and Day on uniformly bounded group representations, Kadison asked whether every bounded representation of a $C^*$-algebra on a Hilbert space is necessarily similar to a $^*$-representation, \cite{K:AJM}. Using Kirchberg's equivalence of the similarity and derivation problems, \cite{Ki:JOT}, we recently observed, \cite{CCSSWW:arXiv}, that the similarity problem is equivalent to the continuity 
of commutants.  The arguments of \cite{CSSW:GAFA} also give a local version of this equivalence: a $C^*$-algebra $A$ has the similarity property if the operation of taking commutants is continuous at $A$, uniformly over all representations of $A$ (see \cite{CCSSWW:arXiv} for the precise statement).  The following consequence is of particular relevance here (we restrict to II$_1$ factors, where it suffices to consider normal representations in the similarity property, see the proof of \cite[Theorem 2.3]{C:JOT}).

\begin{proposition}
Every strongly Kadison-Kastler stable II$_1$ factor satisfies the similarity property.
\end{proposition}

The similarity problem is known to have positive answers for von Neumann algebras of types I$_\infty$, II$_\infty$ and III (see \cite{H:Ann}) but remains open for finite algebras and in particular for factors of type II$_1$. Here, the only factors for which a positive answer is known are those with Murray and von Neumann's property gamma (the gamma factors are those containing asymptotically centralizing sequences, and this property was introduced in \cite{MvN:4} in order to distinguish the hyperfinite II$_1$ factor from the free group factors), \cite{C:JOT}. In particular \emph{McDuff factors} (those factors $M$ which absorb the hyperfinite II$_1$ factor $R$ tensorially, meaning that $M\cong M\vnotimes R$) have the similarity property.  Thus, to produce new examples of strongly Kadison-Kastler stable factors, we work with II$_1$ factors with property gamma.

The role played by the similarity property in obtaining examples of strongly Kadison-Kastler stable factors is encapsulated in the following result, which dates back to \cite{C:IJM}.

\begin{proposition}\label{P-SP}
Let $A$ be a $C^*$-algebra satisfying the similarity property and suppose that $\theta_1,\theta_2:A\rightarrow\BHs$ are two $^*$-representations with $\|\theta_1-\theta_2\|$ sufficiently small.  Then there exists a unitary $u$ on $\Hs$ such that $\theta_2=\Ad(u)\circ\theta_1$. Further, one can control $\|u-\id_\Hs\|$ in terms of $\|\theta_1-\theta_2\|$ and quantitative estimates on how well $A$ satisfies the similarity property.
\end{proposition}
In the presence of the similarity property, if we can show that two close von Neumann algebras $M$ and $N$ on $\Hs$ are $^*$-isomorphic via an isomorphism $\theta$ close to the inclusion map $M\hookrightarrow\BHs$, then it will follow that $\theta$ is spatially implemented by a unitary close to $\id_\Hs$. Consequently $M$ will be strongly Kadison-Kastler stable.

\subsection*{Twisted crossed products}

Our new examples of strongly Kadison-Kastler stable factors arise from the crossed product construction which goes back to Murray and von Neumann.  Consider a countable infinite discrete group $\Gamma$ acting by measure preserving transformations on a probability space $(X,\mu)$ and write $\alpha$ for the induced action of $\Gamma$ on the abelian von Neumann algebra $L^\infty(X)$. A unitary-valued normalized $2$-cocycle is a function $\omega:\Gamma\times\Gamma\rightarrow\mathcal U(L^\infty(X))$ with $\omega(g,e)=\omega(e,g)=1_{L^\infty(X)}$ for all $g\in\Gamma$ which satisfies the cocycle identity
$$
\alpha_g(\omega(h,k))\omega(gh,k)^*\omega(g,hk)\omega(g,h)^*=1_{L^\infty(X)},\quad g,h,k\in\Gamma.
$$
Two such $2$-cocycles $\omega_1,\omega_2$ are \emph{cohomologous} if there exists $\nu:\Gamma\rightarrow\mathcal U(L^\infty(X))$ with $\nu(e)=1_{L^\infty(X)}$ and 
\begin{equation}\label{E-CoB}
\omega_2(g,h)=\left(\alpha_g(\nu(h))\nu(gh)^*\nu(g)\right)\omega_1(g,h),\quad g,h\in\Gamma.
\end{equation}

Given a unitary-valued normalized $2$-cocycle $\omega$, the twisted crossed product
$$
L^\infty(X)\rtimes_{\alpha,\omega}\Gamma
$$
is a von Neumann algebra generated by a copy of $L^\infty(X,\mu)$ and unitaries $(u_g)_{g\in\Gamma}$ satisfying:
\begin{equation}\label{E-TCP1}
u_gfu_g^*=\alpha_g(f),\quad u_gu_h=\omega(g,h)u_{gh},\quad f\in L^\infty(X),\ g,h\in\Gamma
\end{equation}
As the action is measure preserving, we obtain a trace $\tau$ on the twisted crossed product by extending 
\begin{equation}\label{E-TCP2}
\tau(\sum_{g\in\Gamma} f_gu_g)=\int f_{e}\,\mathrm{d}\mu
\end{equation}
from the dense $^*$-subalgebra of finite linear combinations $\sum_{g\in\Gamma} f_gu_g$ with $f_g\in L^\infty(X,\mu)$ so the twisted crossed product is of type II$_1$. The two conditions (eq.~2) and (eq.~3) characterize twisted crossed products, and we will use these to recognize factors close to a twisted crossed product as again of this form, albeit via a possibly different $2$-cocycle.

We will impose two further conditions on the action $\Gamma\curvearrowright X$ in addition to preserving a standard probability measure.
\begin{enumerate}
\item Essential freeness: For $g\neq e$, the stabilizer $\{x\in X:g\cdot x=x\}$ is required to be null. This ensures that the copy of $L^\infty(X)$ is a maximal abelian subalgebra of the twisted crossed product $L^\infty(X)\rtimes_{\alpha,\omega}\Gamma$. 
\item Ergodicity: This requires any $\Gamma$-invariant subset to be either null or co-null. In the presence of freeness, the twisted crossed product $L^\infty(X)\rtimes_{\alpha,\omega}\Gamma$ is a factor if and only if the action is ergodic.
\end{enumerate}
Combining these assumptions, the twisted crossed products $L^\infty(X)\rtimes_{\alpha,\omega}\Gamma$ are always II$_1$ factors.

We are now in position to state our main result. Recall that $SL_n(\mathbb Z)$ denotes the group of $n\times n$ matrices with integer entries and having determinant equal to 1.

\begin{theorem}\label{T-Main}
Let $(X,\mu)$ be a standard probability space and suppose that $SL_n(\mathbb Z)$ acts freely and ergodically by measure preserving transformations on $(X,\mu)$ for $n\geq 3$. Then the II$_1$ factor
\begin{equation}\label{T-Main:Form}
M=(L^\infty(X,\mu)\rtimes_\alpha SL_n(\mathbb Z))\vnotimes R
\end{equation}
is strongly Kadison-Kastler stable.
\end{theorem}

The key property of the group $SL_n(\mathbb Z)$ used in the proof of Theorem \ref{T-Main} is cohomological. By combining the work of Burger and Monod, \cite{BM:JEMS,BM:GAFA}, and Shalom and Monod, \cite{MS:JDG}, with later results of Monod, \cite{M:Crelle}, it follows that the bounded cohomology groups
$$
H^2_b(SL_n(\mathbb Z),L^\infty_{\mathbb R}(X,\mu))
$$
vanish for $n\geq 3$ (a key difficulty which is overcome in \cite{M:Crelle} is that the module $L^\infty_{\mathbb R}(X,\mu)$ is a non-separable Banach space). In Theorem \ref{T-Main}, the groups $SL_n(\mathbb Z)$ can be replaced by any discrete group $\Gamma$ for which $H^2_b(\Gamma,L^\infty_{\mathbb R}(X,\mu))=0$; the work \cite{BM:JEMS,BM:GAFA,MS:JDG,M:Crelle} also establishes this for certain other irreducible higher rank lattices.  The effect of the vanishing of this bounded cohomology group is that the open mapping theorem gives a constant $K>0$ with the property that, for any two unitary $2$-cocycles $\omega_1,\omega_2:\Gamma\times\Gamma\rightarrow\mathcal U(L^\infty(X))$ with
$$
\sup_{g,h\in\Gamma}\|\omega_1(g,h)-\omega_2(g,h)\|<\sqrt{2},
$$
 we can find $\nu:\Gamma\rightarrow\mathcal U(L^\infty(X))$ such that (eq. 1) holds and
$$
\sup_{g\in\Gamma}\|\nu(g)-1_{L^\infty(X)}\|\leq K\sup_{g,h\in\Gamma}\|\omega_1(g,h)-\omega_2(g,h)\|.
$$
For the purpose of finding examples to which Theorem \ref{T-Main} applies, it is useful to note that for measure preserving actions of $SL_n(\mathbb Z)$ with $n\geq 3$ on non-atomic standard probability spaces $(X,\mu)$, ergodicity implies freeness by \cite{SZ:Ann}.  

Examples of suitable actions of $\Gamma=SL_n(\mathbb Z)$ are given by Bernoulli shifts.  Given a base probability space $(Y,\nu)$ (which could be atomic, but is not a singleton) form the infinite product space $X=\prod_{g\in\Gamma}Y$ indexed by the group and let $\mu$ be the product measure on $X$. Then $\Gamma$ acts on $X$ by shifting the indices: $h\cdot (x_g)_{g\in\Gamma}=(x_{hg})_{g\in\Gamma}$. When $\Gamma$ is infinite, this induces a free, ergodic, probability measure preserving action.  By suitably varying the base space $(Y,\nu)$ and using results of Bowen and Popa \cite{Bo:JAMS,P:Invent2,P:IMRN} one obtains an uncountable family of pairwise non-isomorphic factors of the form (eq. 4) to which Theorem \ref{T-Main} applies.

The role of the hyperfinite II$_1$ factor $R$ in Theorem \ref{T-Main} is to ensure that the tensor product $(L^\infty(X,\mu)\rtimes_\alpha SL_n(\mathbb Z))\vnotimes R$ has the similarity property. Indeed, if one could construct a free ergodic probability measure preserving action $\alpha:SL_n(\mathbb Z)\curvearrowright (X,\mu)$ for $n\geq 3$ such that the resulting crossed product factor $L^\infty(X,\mu)\rtimes_\alpha SL_n(\mathbb Z)$ has the similarity property then this crossed product will be strongly Kadison-Kastler stable. However, the only known method for establishing the similarity property for a II$_1$ factor is to establish property gamma. By combining results from \cite{AW:CJM,HJ:OAMP}, the presence of Kazhdan's property (T), \cite{Kaz:FAP}, for $SL_n(\mathbb Z)$ ($n\geq 3$) provides an obstruction to property gamma for the crossed product factors $L^\infty(X,\mu)\rtimes_\alpha SL_n(\mathbb Z)$.

\section*{Outline of the proof of Theorem \ref{T-Main}}

In the light of Proposition \ref{P-SP}, to prove Theorem \ref{T-Main} it suffices to show that if $N$ is close to a II$_1$ factor $M$ of the form (eq. 4), then there is a $^*$-isomorphism of $M$ onto $N$ which is close to the inclusion of $M$ into the containing $\BHs$.  Our strategy for doing this involves three main steps.
\begin{enumerate}[(i)]
\item As $M$ takes the form $M_0\vnotimes R$ (where $M_0=L^\infty(X)\rtimes_\alpha\Gamma$ and $R$ is the hyperfinite II$_1$ factor) we show that $N$ is also a McDuff factor and that after a small unitary perturbation the factorizations of $M$ and $N$ are compatible. To do this, we use the spatial embedding theorem to produce a small unitary perturbation $N_1$ of $N$ which contains $R$, and then define $N_0=(R'\cap N_1)$. One can check that  $d(M_0,N_0)$ is small.  To identify $N_1$ as $N_0\vnotimes R$ we need to show that $N_1$ is generated by $N_0$ and $R$. 

\item To obtain an isomorphism between $M_0$ and $N_0$, we transfer the crossed product structure of $M_0$ to $N_0$. Given a II$_1$ factor $N_0$ which is sufficiently close to a crossed product factor $M_0=L^\infty(X)\rtimes_\alpha\Gamma$, it is possible to  use Theorem \ref{T-C} repeatedly to find a copy of $L^\infty(X)$ inside $N_0$ close to the copy in $M_0$ and unitaries $v_g\in N_0$ normalizing $L^\infty(X)$ and inducing the same action as the $u_g$'s.  We must then show that $N_0$ is generated by $L^\infty(X)$ and the unitaries $v_g$. Once this is achieved, it follows that $N_0$ is a twisted crossed product
$$
L^\infty(X)\rtimes_{\alpha,\omega}\Gamma,
$$
where $\omega$ is a $2$-cocycle measuring the failure of multiplicitivity of the map $g\mapsto v_g$.   

\item\label{S3} In the previous step each $v_g$ can be chosen close to the corresponding $u_g$ so that
$$
\omega(g,h)=v_gv_hv_{gh}^*\approx u_gu_hu_{gh}^*=1_{L^\infty(X)},\quad g,h\in\Gamma.
$$
Our cohomological assumption then ensures that $\omega$ is cohomologous to a trivial cocycle, and this induces a $^*$-isomorphism between $M_0$ and $N_0$.  Moreover, the fact that we ask for the bounded cohomology group $H^2_b(\Gamma,L^\infty_{\mathbb R}(X))$ to vanish (and not just for $H^2(\Gamma,L^\infty_{\mathbb R}(X))$ to vanish) gives additional information: one can find a surjective $^*$-isomorphism $\theta:M_0{\rightarrow}N_0$ such that $\|\theta(fu_g)-fu_g\|$ is small for all $f\in L^\infty(X)$ with $\|f\|\leq 1$ and all $g\in\Gamma$.  In general, there is no reason to expect  $\|\theta(y)-y\|$ to be uniformly small for all $y$ in the unit ball of $M_0$, but we are able to use extra ingredients to achieve this.
\end{enumerate}
A common feature of the first two steps is the need to show that if we are given close von Neumann algebras, one of which is generated by a certain collection of elements, then the second can be generated by suitably chosen elements close to the original generators.  Since the set of generators of a von Neumann algebra is not open in the norm topology, we approach this problem indirectly by changing representations to standard position and working at the Hilbert space level. This is the subject of the next two sections, and the techniques developed there are also used to ensure that $\theta(y)$ is uniformly close to $y$ across the unit ball of $M_0$ in step (\ref{S3}). 

The steps above can be used to prove further stability results; we give two examples.   In Theorem \ref{T-Free}, we use the fact that free groups have cohomological dimension one, so that $H^2(\mathbb F_r,L^\infty_{\mathbb R}(X,\mu))=0$. This enables us to untwist the cocycle $\omega$ in step (\ref{S3}), but since $H^2_b(\mathbb F_r,\mathbb R)\neq 0$, we cannot obtain any information about how the resulting isomorphism behaves on the canonical unitaries.  In Theorem \ref{T-Hyp}, cohomological methods do not apply, and instead we use the recent work of Popa and Vaes, \cite{PV:arXiv2}, on the uniqueness (up to unitary conjugacy) of the Cartan masa in a crossed product by a hyperbolic group. The results of \cite{PV:arXiv2} are valid for a more general class of groups, and Theorem \ref{T-Hyp} holds for this class. 

\begin{theorem}\label{T-Free}
Suppose that $\mathbb F_r\curvearrowright (X,\mu)$ is a free ergodic measure preserving action of a free group on a standard probability space.  Write $M=L^\infty(X)\rtimes_\alpha\mathbb F_r$. Then there exists $\delta>0$ such that if $M\subseteq\BHs$ is a normal unital representation of $M$ and $N\subseteq\BHs$ is a von Neumann algebra with $d(M,N)<\delta$, then $N\cong M$.  If, in addition, we assume that the action is not strongly ergodic (i.e. every sequence of asymptotically invariant subsets of $X$ is approximately null or conull) then such an isomorphism $N\cong M$ is necessarily spatial.
\end{theorem}

\begin{theorem}\label{T-Hyp}
There exists $\delta>0$ with the following property. Suppose that $\Gamma_i\curvearrowright (X_i,\mu_i)$ for $i=1,2$ are two free, ergodic probability measure preserving actions of hyperbolic groups on standard probability spaces and write $M_i=L^\infty(X_i)\rtimes\Gamma_i$.  If $d(M_1,M_2)<\delta$, then $M_1\cong M_2$.
\end{theorem}

\subsection*{Changing representations, standard position and the basic construction}

The theory of normal representations of von Neumann algebras is easy to describe; any two faithful normal representations of a von Neumann algebra are unitarily equivalent after an amplification. Thus, given faithful unital normal representations $\pi_1:M\rightarrow \mathcal B(\Hs_1)$ and $\pi_2:M\rightarrow \mathcal B(\Hs_2)$, we can find a Hilbert space $\Ks$ and a unitary isomorphism $U:\Hs_1\otimes\Ks\rightarrow \Hs_2\otimes\Ks$ such that $U(\pi_1(x)\otimes\id_\Ks)=(\pi_2(x)\otimes\id_\Ks)U$ for all $x\in M$.  In this way, representations of a II$_1$ factor $M$ with separable predual on a separable Hilbert space are classified up to unitary equivalence by the coupling constant or $M$-dimension of the space.  Suppose that $M\subseteq\BHs$ is a unital normal representation on a separable Hilbert space. The commutant $M'$ is a type II factor, so is either type II$_\infty$, where we define $\dim_M(\Hs)=\infty$, or type 
II$_1$, in which case we define $\dim_M(\Hs)=\tau_{M'}(e^M_\xi)/\tau_M(e^{M'}_\xi)$, where $\tau_M$ and $\tau_{M'}$ are the normalized traces on $M$ and $M'$, $\xi$ is a unit vector in $\Hs$, $e^M_\xi$ is the projection in $M'$ onto $\overline{M\xi}$ and $e^{M'}_\xi$ is the projection in $M$ onto $\overline{M'\xi}$. This quantity is independent of the choice of $\xi$. In the lemma below, when $M$ and $N$ have separable preduals we can always reduce to the situation where they act on a separable Hilbert space by cutting by a projection with range $\overline{(M\cup N)''\xi}$ for some $\xi\in\Hs$ which lies in $M'\cap N'$.

\begin{lemma}\label{L7}
Suppose that $M$ and $N$ are II$_1$ factors acting on a separable Hilbert space $\Hs$ with $d(M,N)$ small. Let $\pi_M:M\rightarrow\BKs$ be a unital normal representation on another separable Hilbert space. Then there exists a unital normal representation $\pi_N:N\rightarrow\BKs$ with $d(\pi_M(M),\pi_N(N))\leq O(d(M,N)^{1/2})$. This estimate can be improved to $d(\pi_M(M),\pi_N(N))\leq O(d(M,N))$ when $M$ has the similarity property.
\end{lemma}

\begin{proof}[Sketch proof of Lemma \ref{L7}.] We can assume that $\dim_M(\Hs)=\infty$, as if this is not the case we can  simultaneously amplify both $M$ and $N$ (that is replace $\Hs$ by $\Hs\otimes\ell^2(\mathbb N)$, $M$ by $M\otimes \id_{\ell^2(\mathbb N)}$ and $N$ by $N\otimes \id_{\ell^2(\mathbb N)}$) to reach this situation without changing the distance between $M$ and $N$.  If $\dim_{\pi_M(M)}(\Ks)=\infty$, then $\pi_M$ is unitarily equivalent to the initial representation of $M$ on $\Hs$, and we can use a unitary implementing this equivalence to define $\pi_N$. Otherwise we can find a projection $e\in M'$ such that $x\mapsto xe$ is a unital normal representation of $M$ on $e(\Hs)$ which is unitarily equivalent to $\pi_M$.  When $M$ has the similarity property, $M'$ and $N'$ are close and so $e$ is close to a projection $f$ in $N'$.  We can then find a unitary $u$ close to $\id_{\Hs}$ with $ueu^*=f$. This gives us a normal unital representation of $N$ on $e(\Hs)$ by $y\mapsto u^*yue$ for $y\in N$ and $uNu^*e$ is close to $Me$ on $e(\Hs)$. We define $\pi_N$ by conjugating the representation $y\mapsto u^*yuw$ by the same unitary used to show that $x\mapsto xe$ is equivalent to $\pi_M$.

In the case that $M$ does not have the similarity property, after the initial amplification it will not always be possible to approximate an arbitrary projection in $M'$ by a projection in $N'$. However, using work on the derivation problem in the presence of a cyclic vector which dates back to \cite{C:Scand}, we can show that given $e\in M'$ such that $M$ has a cyclic vector for $e(\Hs)$, then it is possible to find a non-zero subprojection $p\leq e$ in $M'$ which is close to $N'$.  By choosing a projection in $N'$ close to $p$, we obtain close representations of $M$ and $N$ on $p(\Hs)$ as above.  At this point in the argument we are only able to obtain estimates of the form $O(d(M,N)^{1/2})$ in contrast with the $O(d(M,N))$ one obtains in the presence of the similarity property.  Our methods do not enable us to get a lower bound on $\dim_M(p(\Hs))$ which could be very small, but we can take a further subprojection of $p$ to ensure that $\dim_{Mp}(p(\Hs))=\dim_{\pi_M(M)}(\Ks)/n$ for some $n\in\mathbb N$.  In this way, we can make a suitable amplification of our representations on $p(\Hs)$ so that the resulting representation of $M$ is unitarily equivalent to $\pi_M$.
\end{proof}

A II$_1$ factor $M$ is said to be in standard position on a Hilbert space $\Ks$ if $\dim_M(\Ks)=1$. In this case, there exists a unit vector $\xi\in\Ks$ so that the vector state $\langle \cdot \xi,\xi\rangle$ restricts to the traces on $M$ and $M'$. This vector has the properties that $x\xi=0$ for $x\in M$ implies that $x=0$ ($\xi$ is separating for $M$) and that $M\xi$ is dense in $\Ks$ ($\xi$ is cyclic for $M$). These properties also hold for $M'$. One defines the modular conjugation operator $J_M$ with respect to $\xi$ by extending the map $x\xi\mapsto x^*\xi$ for $x\in M$ to a conjugate linear isometry on $\Ks$. The commutant $M'$ takes the form $J_MMJ_M$ and so we have an anti-isomorphism $x\mapsto J_MxJ_M$ between $M$ and $M'$. 

By applying Lemma 9 to a pair of close II$_1$ factors $M$ and $N$ on $\Hs$, we can find new close representations on a Hilbert space $\Ks$ where $M$ is now in standard position. Our objective is to show that $N$ is also in standard position on $\Ks$.  To do this we first extend the work of \cite[Section 3]{CSSW:GAFA} to show that $N$ is almost in standard position in that $\dim_N(\Ks)\approx 1$, whence it follows that $M'$ and $N'$ are close on $\Ks$ (this is automatic when $M$ has the similarity property).  Now given an amenable subalgebra $P\subseteq M$, we have $P\subset_\gamma N$ and $J_MPJ_M\subset_\gamma N'$ for some small $\gamma$, and so we can use the spatial embedding theorem (Theorem \ref{T-C}) twice to replace $N$ by a small unitary perturbation such that $P\subseteq N$ and $J_MPJ_M\subseteq N'$.  In this way we can apply the next lemma to see that $N$ is in standard position.

\begin{lemma}\label{L-Std}
Suppose that $M$ is a II$_1$ factor in standard position on $\Ks$ with respect to $\xi\in \Ks$ and suppose that $A\subseteq M$ is a maximal abelian subalgebra (masa) in $M$.  Suppose that $N$ is another II$_1$ factor on $\Ks$ such that $A\subseteq N$,  $J_MAJ_M\subseteq N'$ and $d(M,N)$ is sufficiently small.  Then $\xi$ is a tracial vector for $N$ and $N'$ so $N$ is also in standard position on $\Ks$.
\end{lemma}

\begin{proof}[Sketch proof of Lemma \ref{L-Std}.] This is proved by using the unique trace preserving expectation $E^N_A$ from $N$ onto $A$.  It is easy to check that as $A$ is maximal abelian in $M$ it is also maximal abelian in $N$, and then the form of $E^N_A$ is known: for $x\in N$, $E^N_A(x)$ lies in the strong$^*$-closed convex hull of the set $\{uxu^*:u\in\mathcal U(A)\}$ of unitary conjugates of $x$ by $A$.  The assumption $J_MAJ_M\subseteq N'$ gives
$$
\langle uxu^*\xi,\xi\rangle=\langle xJ_MuJ_M\xi,J_MuJ_M\xi\rangle=\langle x\xi,\xi\rangle,\quad u\in\mathcal U(A),
$$
so that $\langle E^N_A(x)\xi,\xi\rangle=\langle x\xi,\xi\rangle$ for all $x\in N$. As $E^N_A(x)\in A\subseteq M$ and $\xi$ is tracial for $M$, we have $\tau_M(E^N_A(x))=\langle x\xi,\xi\rangle$.  However it is not hard to check that as $M$ and $N$ are close, $\tau_M$ and $\tau_N$ agree on $A$ so that $\tau_N(E^N_A(x))=\langle x\xi,\xi\rangle$ for $x\in N$.  Since $E^N_A$ is $\tau_N$-preserving this shows that $\xi$ is tracial for $N$.  

To see that $\xi$ is also tracial for $N'$, interchange the roles of the algebras $M$ and $N$ and 
their commutants. Here we use the standard position of $M$ to ensure that $d(M',N')$ is small.
\end{proof}

In fact we immediately get further information: in the situation of Lemma \ref{L-Std} the inclusions $A\subseteq M$ and $A\subseteq N$ induce the same basic construction. This construction, developed extensively in \cite{J:Invent} is the starting point for Jones's theory of subfactors, and plays a key role in perturbation results for subalgebras of finite von Neumann algebras, \cite{P:Invent2,C:MA}.  Given a subalgebra $A$ of $M$ write $e_A$ for the projection on $\Ks$ with range $\overline{A\xi}$. The basic construction of $A\subseteq M$ is the von Neumann algebra $(M\cup\{e_A\})''$ obtained by adjoining $e_A$ to $M$ and is denoted $\langle M,e_A\rangle$.  This satisfies $\langle M,e_A\rangle=(J_MAJ_M)'$.
\begin{corollary}\label{C-Basic}
With the same hypotheses as in Lemma \ref{L-Std}, we have
$$
\langle M,e_A\rangle=\langle N,e_A\rangle.
$$
\end{corollary}
\begin{proof}[Proof of Corollary \ref{C-Basic}.]
We have $J_MAJ_M\subseteq N'=J_NNJ_N$ by hypothesis. Standard properties of the basic construction from \cite{J:Invent} show that $e_A$ commutes with $A$ and $J_M$ so that $J_MAJ_M\subseteq J_NNJ_N\cap\{e_A\}'=J_N(N\cap \{e_A\}')J_N=J_NAJ_N$ (using the fact that $N\cap \{e_A\}'=A$) so that 
$$
J_MAJ_M\subseteq J_NAJ_N\subseteq J_NNJ_N=N'.
$$
Now $J_MAJ_M$ is a masa in $M'$, moreover $M'$ and $N'$ are close, whence it follows that $J_MAJ_M$ is also maximal abelian in $J_NNJ_N=N'$. Hence $J_MAJ_M=J_NAJ_N$, and the result follows by taking commutants. 
\end{proof}

Once we have reached this point of our argument, we can replace $A$ in Corollary \ref{C-Basic} by an amenable subalgebra $P\subseteq M$ with $P'\cap M\subseteq P$ using a technical theorem of Popa from \cite{P:Invent}. This enables us to formulate versions of our main results for suitable actions of discrete groups on the hyperfinite II$_1$ factor: any factor of the form $(R\rtimes_\alpha SL_n(\mathbb Z))\vnotimes R$ for a properly outer action $\alpha$ and $n\geq 3$ is strongly Kadison-Kastler stable.

\subsection*{Using the basic construction to prove Theorem \ref{T-Main}}

A considerable amount of information regarding an inclusion $A\subseteq M$ of finite von Neumann algebras is encoded in the basic construction algebra $\langle M,e_A\rangle$. Of particular relevance here is Popa's result \cite[Proposition 1.4.3]{P:Ann} which shows that a masa $A$ in a II$_1$ factor $M$ is \emph{Cartan} in the sense of \cite{D:Ann} (i.e. the group of normalizers $\mathcal N_M(A)=\{u\in\mathcal U(M):uAu^*=A\}$ generates $M$ as a von Neumann algebra) if and only if $A'\cap \langle M,e_A\rangle$ is generated by projections which are finite in $\langle M,e_A\rangle$.  As the spatial embedding theorem, Lemma \ref{L7}, Lemma \ref{L-Std} and Corollary \ref{C-Basic} combine to show that close inclusions of masas into II$_1$ factors can be adjusted via a small unitary perturbation to give the same basic construction algebras, (albeit possibly on a different Hilbert space) we obtain the next result.

\begin{proposition}
Let $A\subseteq M$ be a Cartan masa in a II$_1$ factor acting on a Hilbert space $\Hs$. Any inclusion $B\subseteq N$ with $d(M,N)$ and $d(A,B)$ sufficiently small is also an inclusion of a Cartan masa in a II$_1$ factor. 
\end{proposition} 
 
Given a crossed product II$_1$ factor $M_0=L^\infty(X)\rtimes\Gamma$ arising from a free ergodic probability measure preserving action $\alpha:\Gamma\curvearrowright (X,\mu)$, and another factor $N_0$ close to $M_0$, the assumption of freeness ensures that $L^\infty(X)$ is a maximal abelian subalgebra of $M_0$. In step 2 of Theorem \ref{T-Main}, we use the spatial embedding theorem to assume that $A=L^\infty(X)\subset N_0$ and to find unitary normalizers $\{v_g\}_{g\in\Gamma}$ in $N_0$ close to the canonical unitary normalizers $\{u_g\}_{g\in\Gamma}$ in $M_0$.  The previous proposition shows that $N_0$ is generated by all normalizers of $A$ but in fact $N_0$ is generated by $A\cup\{v_g:g\in\Gamma\}$ as required for step 2 of the proof of Theorem \ref{T-Main}.  Once we convert to standard position so that $A\subset M_0$ and $A\subset N_0$ induce the same basic construction, one sees this by first noting that $\{u_ge_Au_g^*\}_{g\in\Gamma}$ are pairwise orthogonal and sum to $1_{M_0}=1_{N_0}$. Since $u_g$ and $v_g$ are close, we must have $v_ge_Av_g^*\approx u_ge_Au_g^*$, but in fact these projections are equal (since they both lie in $Z(A'\cap\langle M,e_A\rangle)$. The fact that $1_N=\sum_{g\in\Gamma}v_ge_Av_g^*$ can then be used to see that finite linear combinations $\sum_{g\in\Gamma}v_gf_g$ (for $f_g\in A$) are dense in $N_0$.

A similar argument, working at the Hilbert space level, is used in step 1 to show that if $M$ is a McDuff factor of the form $M_0\vnotimes R$, and $N$ is close to $M$ and contains $R$, then $N$ is generated by the commuting subalgebras $R'\cap N$ and $R$.

The fact that $\sum_{g\in\Gamma}u_ge_Au_g^*=1_{M_0}=1_{N_0}$ in $Z(A'\cap \langle M_0,e_A\rangle)$ is also vital in step 3 of the proof of Theorem \ref{T-Main}.  At this point, using our earlier results, we have a crossed product $M_0=L^\infty(X)\rtimes_\alpha\Gamma$ acting in standard position on $\Hs$ with respect to $\xi$, and an isomorphic copy $N_0$ of $M_0$ on $\Hs$ with $A=L^\infty(X)\subseteq N_0$ and $J_{M_0}AJ_{M_0}=J_{N_0}AJ_{N_0}\subseteq N_0'$. The isomorphism $\theta:M_0\rightarrow N_0$ is obtained from step 2 using the vanishing of the bounded cohomology group $H^2_b(\Gamma,L^\infty_\mathbb R(X))$ and satisfies $\theta(f)=f$ for $f\in A$ and that $\|\theta(u_g)-u_g\|$ is small for all $g\in\Gamma$. Since Lemma \ref{L-Std} shows that $M_0$ and $N_0$ are both in standard position on $\Hs$, the isomorphism $\theta$ is spatially implemented on $\Hs$ by $W$ where $W$ is given by by extending the map $W(x\xi)=\theta(x)\xi$ for $x\in M_0$.  Since $\theta(f)=f$ for $f\in A$, it follows that $W\in A'$, and similarly the assumption that $\theta(M_0)=N_0\subseteq (J_{M_0}AJ_{M_0})'=(J_{N_0}AJ_{N_0})'$ ensures that $W\in (J_{M_0}AJ_{M_0})'=\langle M_0,e_A\rangle$.  That is $W\in A'\cap \langle M_0,e_A\rangle$.

Write $P_g=u_ge_Au_g^*$. It is a standard fact that $(A'\cap\langle M_0,e_A\rangle)P_g=AP_g$ for each $g\in \Gamma$ and $W$ decomposes as $W=\sum_{g\in\Gamma}w_gP_g$ for some unitary operators $w_g\in A$.  For each $g\in\Gamma$ the condition that $\theta(u_g)\approx u_g$ translates to $\alpha_g(w_e)\approx w_g$ so that (using the centrality of $P_g$), $W\approx W_1=\sum_{g\in\Gamma}\alpha_g(w_e)P_g$. On the other hand $J_{M_0}w_eJ_{M_0}P_g=\alpha_g(w_e)P_g$ so $W_1\in J_{M_0}AJ_{M_0}\subseteq (M_0\cup N_0)'$. Thus $\theta=\Ad(W)=\Ad(W_1^*W)$ has $\|\theta-\id_{M_0}\|_{cb}\leq 2\|W-W_1\|$, giving us uniform control on $\|\theta(x)-x\|$ across the unit ball of $M_0$.

\section*{Concluding remarks and open questions}

We end with some questions and possible future directions.  

It is not hard to use Lemma \ref{L-Std} to show that for each $K>0$, there exists $\delta>0$ with the property that if $M,N\subset\BHs$ are II$_1$ factors with $d(M,N)<\delta$ and $\dim_M(\Hs)\leq K$, then $\dim_N(\Hs)=\dim_M(\Hs)$. However we have not been able to show that sufficiently close II$_1$ factors necessarily have the same coupling constant in general. One consequence of a positive answer to this question would be that in Theorem \ref{T-Free} the isomorphism would automatically be spatial without the assumption of a non-strongly ergodic action.

\begin{question}\label{Q1}
Does there exist $\delta>0$ such that whenever $M,N\subset\BHs$ are II$_1$ factors with $d(M,N)<\delta$, then $\dim_M(\Hs)=\dim_N(\Hs)$?
\end{question}

\medskip

In Theorem \ref{T-Hyp}, we use uniqueness results for Cartan masas from \cite{PV:arXiv2} in order to obtain an isomorphism.  In contrast with Theorems \ref{T-Main} and \ref{T-Free}, this method relies on imposing structural hypotheses on both $M$ and $N$.  Further, there are hyperbolic groups $\Gamma$ for which Theorem \ref{T-Hyp} applies, but our cohomological methods do not.  Such factors provide a suitable test case for future developments.

\begin{question}
Let $M=L^\infty(X)\rtimes_\alpha\Gamma$ be a crossed product factor such that $L^\infty(X)$ is the unique Cartan masa up to unitary conjugacy, but the comparison map $$H^2_b(\Gamma,L^\infty_{\mathbb R}(X))\rightarrow H^2(\Gamma,L^\infty_{\mathbb R}(X))$$ is not zero (i.e. there are non-trivial bounded $2$-cocycles which are not trivial in $H^2(\Gamma,L^\infty_{\mathbb R}(X))$).  Does there exist $\delta>0$ such that any II$_1$ factor $N$ with $d(M,N)<\delta$ is isomorphic to $M$?
\end{question}

Is it possible to use the methods of \cite{Roy:arXiv} to find stability results for factors which are `completely close', i.e. $d_{cb}(M,N)=\sup d(\mathbb M_n(M),\mathbb M_n(N))$ is small? 

Finally, what is the analogous statement to Theorem \ref{T-Main} in the category of $C^*$-algebras?  A major difficulty here is that the known embedding theorem for separable nuclear $C^*$-algebras from \cite{HKW:Adv} is not as strong as Theorem \ref{T-C} in that it does not guarantee that the resulting embedding is spatial, and due to the counterexamples of \cite{Joh:CMB} it cannot give uniform control on the embedding.

\subsection*{Acknowledgments}
JC's research is partially supported by an AMS-Simons research travel grant, RS is partially supported by NSF grant DMS-1101403 and SW is partially supported by EPSRC grant EP/I019227/1. The authors gratefully acknowledge the following
additional sources of funding which enabled this research to be undertaken. A visit of EC
to Scotland in 2007 was supported by a grant from the Edinburgh Mathematical
Society; a visit of RS to Scotland in 2011 was supported by a grant from the
Royal Society of Edinburgh; SW visited Vassar college in 2010 supported by the
Rognol distinguished visitor program; JC, RS and SW visited Copenhagen in 2011 supported by the FNU of Denmark.

SW would
like to thank Peter Kropholler and Nicolas Monod for helpful discussions about
group cohomology and Jesse Petersen for helpful discussions about group actions.  The authors would like to thank Wai Kit Chan for his helpful comments
on early drafts of our work.

\end{document}